\documentclass[12pt]{article}
\usepackage{amsmath,amssymb,amsthm,comment,cite}

\newtheorem{theorem}{Theorem}[section]
\newtheorem{thmy}{Theorem}
\newtheorem{lemma}[theorem]{Lemma}

\def\barr{\begin{array}}
\def\earr{\end{array}}

\title{A note on finite groups with few conjugacy classes of subgroups}
\author{Marius T\u arn\u auceanu}
\date{April 29, 2026}

\begin{document}

\maketitle

\begin{abstract}
In this note, we give some conditions of solvability of a finite group $G$ based on its number of conjugacy classes of subgroups $Con(G)$. We classify all finite groups $G$ with $Con(G)\leq 5$. A new characterization of $A_5$ is also given. 
\end{abstract}

{\small
\noindent
{\bf MSC2020\,:} Primary 20D60; Secondary 20D10, 20F16, 20F17.

\noindent
{\bf Key words\,:} conjugacy classes of subgroups, solvable groups.}

\section{Introduction}

In the last years, there has been a growing interest in detecting structural properties of a finite group $G$ through the study of some functions related to elements or to subgroups of $G$ (see e.g. \cite{9}). The starting point for our discussion is given by the paper \cite{7}, which provides some conditions of (super)-solvability and nilpotency of $G$ based on its number of subgroups $Sub(G)$. It also gives a classification of finite groups with few subgroups, generalizing the works by Betz and Nash \cite{4,5}.

In what follows, we will deal with conjugacy classes of subgroups of a finite group $G$. Their total number $Con(G)$, as well as their number for certain types of subgroups, such as maximal subgroups, non-normal subgroups or non-cyclic subgroups, have been studied in many recent papers (see e.g. \cite{6}, \cite{10,11,12,13} and \cite{15}). In our note, we prove first finite groups with few conjugacy classes of subgroups are solvable, more precisely:

\begin{theorem}
Let $G$ be a finite group. If $Con(G)<9$, then $G$ is solvable.
\end{theorem}

It is easy to see that the smallest number of conjugacy classes of subgroups of a non-solvable group is $9$ and it is attained for the alternating group of five symbols. Moreover, $A_5$ is the unique non-solvable group $G$ with $Con(G)=9$.

\begin{theorem}
Let $G$ be a finite non-solvable group. If $Con(G)=9$, then $G\cong A_5$.
\end{theorem}

Then a similar conclusion to that of Theorem 1.1 is obtained for $Con(G)=10$ or $Con(G)=11$.

\begin{theorem}
Let $G$ be a finite group. If $Con(G)\in\{10, 11\}$, then $G$ is solvable.
\end{theorem}

We observe that there are finite non-solvable groups $G$ with $Con(G)=12$, such as ${\rm SL}(2,5)$ or ${\rm PSL}(2,8)$. Clearly, the study can be extended to higher values of $Con(G)$.

Note that we are also able to completely classify finite groups $G$ with a small number of conjugacy classes of subgroups. For example:

\begin{itemize}
\item[-] If $Con(G)\leq 3$, then $G$ is cyclic of type $\mathbb{Z}_{p^i}$, where $p$ is a prime and $i=0,1,2$. 
\item[-] If $Con(G)=4=Con(S_3)$, then $G$ is either cyclic of type $\mathbb{Z}_{p^3}$ or $\mathbb{Z}_{pq}$, where $p\neq q$ are primes, or non-abelian of order $pq$, where $p,q$ are primes and $q\mid p-1$.
\item[-] If $Con(G)=5=Con(A_4)$, then $G$ is either abelian of type $\mathbb{Z}_2^2$ or $\mathbb{Z}_{p^4}$, where $p$ is a prime, or non-CLT of order $p^2q$, where $p,q$ are primes, $q$ is odd and $q\mid p+1$.
\end{itemize}

For the proof of our results, we need the following theorem (see Theorems 1, 2 of \cite{2}, Theorem 1.1 of \cite{3} and Exercise 7, Section 10.5, of \cite{16}).

\begin{thmy}
Given a finite group $G$, we denote by $k(G)$ and $k'(G)$ the numbers of conjugacy classes of maximal subgroups of $G$ and of non-normal maximal subgroups of $G$, respectively. The following hold:
\begin{itemize}
\item[{\rm a)}] If $k'(G)\leq 2$, then $G$ is solvable. In particular, a finite group $G$ with $k(G)\leq 2$ is always solvable.
\item[{\rm b)}] If $k(G)=3$, then $G$ is non-solvable if and only if either $G/\Phi(G)\cong {\rm PSL}(2,7)$ or $G/\Phi(G)\cong {\rm PSL}(2,2^p)$, where $p$ is a prime.
\item[{\rm c)}] If $k(G)=4$ and $G$ is simple, then it isomorphic to one of the groups ${\rm PSL}(2,11)$, ${\rm PSL}(2,p)$ {\rm(}$p$ is a prime,
$p>3$ and $p\equiv \pm 3, \pm 13\, (mod\, 40)${\rm)}, ${\rm PSL}(2,p^{r^m})$ {\rm(}$p$, $r$ are primes, $r>2$ for $p>2$, $m\in\mathbb{N}$ and $pm>2${\rm)}, ${\rm PSL}(3,3)$, $U_3(q)$ {\rm(}$q=3$ or $q=2^{2^m}$, $m\in\mathbb{N}${\rm)} or ${\rm Sz}(2^r)$ {\rm(}$r$ is an odd prime{\rm)}.
\item[{\rm d)}] If $G$ has an abelian maximal subgroup, then $G$ is solvable.
\end{itemize}
\end{thmy}

We also need an immediate consequence of the well-known Zsigmondy's theorem.

\begin{thmy}
Let $a,b\in\mathbb{N}$ such that $\gcd(a, b)=1$ and $n\in\mathbb{N}$, $n>1$. Then there exists a prime divisor of $a^n+b^n$ that does not divide $a^k+b^k$ for all $k\in\{1, \dots, n-1\}$, except for the case $a=2$, $b=1$ and $n=3$.
\end{thmy}

Another result that will be used below is the following description of the structure of finite non-solvable groups with a nilpotent maximal subgroup (see \cite{1}).

\begin{thmy}
Let $G$ be a finite non-solvable group with a nilpotent maximal subgroup. Then $G/F(G)$ has a unique minimal normal subgroup $H/F(G)$, $H/F(G)$ is a direct product of copies of a simple group with dihedral Sylow $2$-subgroups, and $G/H$ is a $2$-group.
\end{thmy}

Finally, we remark that 
\begin{equation}
Con(G)=\frac{1}{|G|}\sum_{H\leq G}|N_G(H)|\nonumber
\end{equation}and so this function is related to the function
\begin{equation}
\sigma_1(G)=\frac{1}{|G|}\sum_{H\leq G}|H|\nonumber
\end{equation}studied in our previous papers \cite{14,17}. This connection suggests the following question.
\bigskip

\noindent{\bf Open problem.} Study other functions of type
\begin{equation}
f(G)=\frac{1}{|G|}\sum_{H\leq G}|g(H)|,\nonumber
\end{equation}where $g:L(G)\longrightarrow L(G)$ is a fixed function; for example, $g(H)=C_G(H)$ or $g(H)=Core_G(H)$, $\forall\, H\leq G$.
\bigskip

Most of our notation is standard and will not be repeated here. Basic definitions and results on groups can be found in \cite{16}.

\section{Proofs of the main results}

We start with an elementary lemma that collects some basic properties of the function $Con$.

\begin{lemma}
    The following hold:
\begin{itemize}
\item[{\rm a)}] If $G$ is a finite group, then $Con(G)\geq|\pi_s(G)|$, where $\pi_s(G)=\{|H|\mid H\leq G\}$.
\item[{\rm b)}] If $G$ is a finite group and $H$ is a normal subgroup of $G$, then $Con(G)\geq Con(G/H)$. Moreover, we have equality if and only if $H=1$. 
\item[{\rm c)}] The function $Con$ is multiplicative, that is if $G_1$ and $G_2$ are finite groups of coprime orders, then $Con(G_1\times G_2)=Con(G_1)Con(G_2)$. In par\-ti\-cu\-lar, if $G$ is a finite nilpotent group of order $p_1^{n_1}\cdots p_r^{n_r}$, then $Con(G)=\prod_{i=1}^rCon(P_i)$, where $P_i$ is a Sylow $p_i$-subgroup of $G$, $\forall\, i=1, \dots, r$.
\end{itemize}
\end{lemma}

Next, we present an arithmetic result that will be used in the following.

\begin{lemma}
Let $p$ be a prime. Then $|{\rm PSL}(2,2^p)|$ is divisible by exactly three primes if and only if either $p=2$ or $p=3$.
\end{lemma}

\begin{proof}
We have $|{\rm PSL}(2,2^p)|=2^p(2^p-1)(2^p+1)$ and this number is divisible by exactly three primes if and only if both $2^p-1$ and $2^p+1$ are prime powers. This holds for $p=2$ and $p=3$. If $p\geq 5$, then $p$ is odd and so $2^p+1$ is divisible by $3$. On the other hand, in this case there exists a prime divisor $q$ of $2^p+1$ that does not divide $2^k+1$ for all $k\in\{1, \dots, p-1\}$ by Theorem B. Then $q\nmid 3$, that is $q\neq 3$, implying that $2^p+1$ cannot be a prime power. Thus the unique possibilities are $p=2$ and $p=3$.\qedhere
\end{proof}

Note that a complete classification of finite simple groups whose order is divisible by exactly three primes follows from \cite{8}. Such a group is one of the following groups:
\begin{itemize}
\item[-] ${\rm PSL}(2,4)\cong A_5$ (of order $60=2^2\cdot 3\cdot 5$),
\item[-] ${\rm PSL}(2,7)$ (of order $168=2^3\cdot 3\cdot 7$),\newpage
\item[-] $A_6$ (of order $360=2^3\cdot 3^2\cdot 5$),
\item[-] ${\rm PSL}(2,8)$ (of order $504=2^3\cdot 3^2\cdot 7$),
\item[-] ${\rm PSL}(2,17)$ (of order $2448=2^4\cdot 3^2\cdot 17$),
\item[-] ${\rm PSL}(3,3)$ (of order $5616=2^4\cdot 3^3\cdot 13$),
\item[-] $U_3(3)$ (of order $6048=2^5\cdot 3^3\cdot 7$),
\item[-] $U_4(2)$ (of order $25920=2^6\cdot 3^4\cdot 5$).
\end{itemize}

We are now able to prove our main results.

\bigskip\noindent{\bf Proof of Theorem 1.1.} Assume that $Con(G)<9$, but $G$ is not sol\-va\-ble. Let $|G|=p_1^{n_1}\cdots p_r^{n_r}$ with $p_1, \dots, p_r$ distinct primes and denote by $P_i$ a Sylow $p_i$-subgroup of $G$, $\forall\, i=1, \dots, r$. Then $r\geq 3$.

If at least two Sylow subgroups of $G$ are maximal subgroups, say $P_1$ and $P_2$, then they cannot be abelian by Theorem A, d), and so $n_1,n_2\geq 3$. It follows that $G$ contains at least nine conjugacy classes of subgroups of orders $1$, $|G|$, $p_3^{n_3}$ and $p_i$, $p_i^2$, $p_i^3$, $i=1,2$, a contradiction. If a unique Sylow subgroup of $G$ is a maximal subgroup, say $P_1$, then $n_1\geq 3$. It follows that $G$ contains at least seven conjugacy classes of subgroups of orders $1$, $|G|$, $p_1$, $p_1^2$, $p_1^3$, $p_2^{n_2}$ and $p_3^{n_3}$. Since, by Theorem A, a), $G$ has at least two conjugacy classes of maximal subgroups different from that determined by $P_1$, we get $Con(G)\geq 7+2=9$, a contradiction. Thus no Sylow subgroup of $G$ is maximal.

Therefore $G$ contains at least eight conjugacy classes of subgroups: three of Sylow subgroups, three of maximal subgroups and two of $1$ and $G$. But $Con(G)\leq 8$, which implies that $k'(G)=k(G)=r=3$. Moreover, the Sylow $p$-subgroups of $G$ cannot have proper subgroups and so $n_i=1$, $i=1,2,3$. Then $G$ is of square-free order and consequently it is solvable, a contradiction. The proof of Theorem 1.1 is now complete.\qed

\bigskip\noindent{\bf Proof of Theorem 1.2.} Assume that $Con(G)=9$. We will prove that $k(G)=3$.

Since $G$ is not solvable, its order is divisible by at least three primes. Also, we have $k(G)\geq 3$ by Theorem A, a). Assume that $k(G)\geq 4$. As in the proof of Theorem 1.1, the Sylow subgroups of $G$ cannot be maximal subgroups, which shows that $G$ contains at least nine conjugacy classes of subgroups: three of Sylow subgroups, four of maximal subgroups and two of $1$ and $G$. But $Con(G)=9$ and so these classes are all conjugacy classes of subgroups of $G$. It follows that the Sylow subgroups of $G$ must be cyclic of prime order and therefore $G$ must be solvable because its order is square-free. This contradicts our assumption.

Since $k(G)=3$, it follows that $G/\Phi(G)\cong {\rm PSL}(2,7)$ or $G/\Phi(G)\cong {\rm PSL}(2,2^p)$, where $p$ is a prime, by Theorem A, b). In the first case Lemma 2.1, b), leads to
\begin{equation}
9=Con(G)\geq Con(G/\Phi(G))=Con({\rm PSL}(2,7))=15,\nonumber
\end{equation}a contradiction. Since $|G|$ and $|G/\Phi(G)|$ have the same prime factors, in the second case we obtain that $|{\rm PSL}(2,2^p)|$ is divisible by exactly three primes. Then Lemma 2.2 shows that either $p=2$ or $p=3$. If $p=2$, we have ${\rm PSL}(2,4)\cong A_5$ and
\begin{equation}
9=Con(G)\geq Con(G/\Phi(G))=Con(A_5)=9,\nonumber
\end{equation}implying that $\Phi(G)=1$ and $G\cong A_5$. If $p=3$, we have
\begin{equation}
9=Con(G)\geq Con(G/\Phi(G))=Con({\rm PSL}(2,8))=12,\nonumber
\end{equation}a contradiction. 

Thus the unique possibility is $G\cong A_5$, as desired.\qed

\bigskip\noindent{\bf Proof of Theorem 1.3.} Assume that $Con(G)\in\{10, 11\}$, but $G$ is not solvable. Denote by $[M_1], \dots, [M_{k(G)}]$ the conjugacy classes of maximal subgroups of $G$, where $k(G)\geq 3$ from Theorem A, a). Let $|G|=p_1^{n_1}\cdots p_r^{n_r}$ with $p_1, \dots, p_r$ distinct primes and let $P_i$ be a Sylow $p_i$-subgroup of $G$, $\forall\, i=1, \dots, r$. Then $r\geq 3$. We may assume that $n_1\geq 2$, and even that $n_1\geq 3$ if $P_1$ is a maximal subgroup of $G$ by Theorem A, d). We distinguish the following cases:

\medskip

\hspace{5mm}{\bf Case 1.} $k(G)\geq 4$\\
We have two possibilities:

\medskip

\hspace{10mm}{\bf Subcase 1.1.} $r\geq 4$\\ 
Suppose that at least one of the inequalities $k(G)\geq 4$ and $r\geq 4$ is strict. Then $k(G)+r\geq 9$ and therefore $G$ has at least $12$ conjugacy classes of subgroups
\begin{itemize}
\item[-] $[G]$, $[M_1]$, \dots, $[M_{k(G)}]$, $[P_1]$, \dots, $[P_r]$, $[R_1]$ and $[1]$, where $R_1\leq P_1$ with $|R_1|=p_1$, if $P_1$ is not a maximal subgroup of $G$,
\item[-] $[G]$, $[M_1]$, \dots, $[M_{k(G)}]$, $[P_2]$, \dots, $[P_r]$, $[Q_1]$, $[R_1]$, and $[1]$, where $R_1,Q_1\leq P_1$ with $|R_1|=p_1$ and $|Q_1|=p_1^2$, if $P_1$ is a maximal subgroup of $G$, say $P_1=M_1$,
\end{itemize}a contradiction. Thus $k(G)=r=4$ and it follows that $G$ has exactly $11$ conjugacy classes of subgroups, namely
\begin{itemize}
\item[-] $[G]$, $[M_1]$, \dots, $[M_4]$, $[P_1]$, \dots, $[P_4]$, $[R_1]$ and $[1]$, where $R_1\leq P_1$ with $|R_1|=p_1$, if $P_1$ is not a maximal subgroup of $G$,
\item[-] $[G]$, $[M_1]$, \dots, $[M_4]$, $[P_2]$, \dots, $[P_4]$, $[Q_1]$, $[R_1]$, and $[1]$, where $R_1,Q_1\leq P_1$ with $|R_1|=p_1$ and $|Q_1|=p_1^2$, if $P_1$ is a maximal subgroup of $G$, say $P_1=M_1$.
\end{itemize}This shows that either $|G|=p_1^2p_2p_3p_4$ or $|G|=p_1^3p_2p_3p_4$. Clearly, we have $p_1=2$. Moreover, in the second situation we have $P_1\cong Q_8$ because $P_1$ is non-abelian and $D_8$ possesses two types of subgroups of order $4$, and this contradicts Theorem C. In the first situation, assume that $G$ has a proper non-trivial normal subgroup $H$. Then one of the groups $H$ and $G/H$ must be non-solvable and so it must be a simple group of order $4uv$, where $u,v\in\{p_2,p_3,p_4\}$ are distinct. Since $A_5$ is the unique such group, it follows that either $H\cong A_5$ or $G/H\cong A_5$. In both cases we obtain $G\cong A_5\times C_w$ with $w\in\{p_2,p_3,p_4\}$, implying that $Con(G)=18$, a contradiction. 

Thus $G$ is a simple group and consequently it is one of the groups in Theorem A, c). Among these groups, the only ones having a Sylow $2$-subgroup isomorphic to $\mathbb{Z}_2^2$ are ${\rm PSL}(2,p)$, where $p>3$ is a prime sa\-tis\-fy\-ing $p\equiv 3, 13\, ({\rm mod}\, 40)$. Then $3$ divides the order of $G$ and we may assume that $p_2=3$. Since $|{\rm PSL}(2,p)|=\frac{p(p^2-1)}{2}$\,, we may also assume that $p=p_3$. It follows that $p_3^2-1=24p_4$, an equation which has the unique solution $p_3=11$ and $p_4=5$. Thus $G\cong{\rm PSL}(2,11)$, implying that $Con(G)=16$, a contradiction.

\medskip

\hspace{10mm}{\bf Subcase 1.2.} $r=3$\\ 
Then $G$ has a section $S=H/K$ isomorphic to one of the groups $A_5$, ${\rm PSL}(2,7)$, $A_6$, ${\rm PSL}(2,8)$, ${\rm PSL}(2,17)$, ${\rm PSL}(3,3)$, $U_3(3)$ or $U_4(2)$. Since
\begin{equation}
|\pi_s(S)|\leq|\pi_s(G)|\leq Con(G)\leq 11,\nonumber
\end{equation}it is easy to see that the unique possibility is $S\cong A_5$, in which case we have $|\pi_s(S)|=9$.

If $H\neq G$, then $H$ is contained in a maximal subgroup of $G$, say $H\leq M_1$. It follows that $G$ has at least $9+k(G)\geq 13$ conjugacy classes of subgroups: the $9$ that are contained in $[H]$ and $[G]$, $[M_2]$, \dots, $[M_{k(G)}]$, a contradiction. 

If $H=G$, then $[K]\neq [1]$ because $G\not\cong A_5$. Also, $[K]$ is contained in exactly $3$ conjugacy classes of maximal subgroups of $G$, say $[M_1]$, $[M_2]$ and $[M_3]$. On the other hand, from $r=3$ we infer that $[K]$ is not contained in at least $2$ conjugacy classes of minimal subgroups of $G$, say $[R_1]$ and $[R_2]$. Thus 
\begin{equation}
Con(G)\geq 9+1+(k(G)-3)+2\geq 13,\nonumber 
\end{equation}contradicting again our hypothesis.

\medskip

\hspace{5mm}{\bf Case 2.} $k(G)=3$\\ 
By Theorem A, b), we have either $G/\Phi(G)\cong {\rm PSL}(2,7)$ or $G/\Phi(G)\cong {\rm PSL}(2,2^p)$, where $p$ is a prime. If $G/\Phi(G)\cong {\rm PSL}(2,7)$, then Lemma 2.1, b), implies that
\begin{equation}
Con(G)\geq Con(G/\Phi(G))=15,\nonumber 
\end{equation}a contradiction. Assume that $G/\Phi(G)\cong {\rm PSL}(2,2^p)$ with $p$ prime. Since ${\rm PSL}(2,2^p)$ possesses at least $7+p$ conjugacy classes of subgroups, namely $[{\rm PSL}(2,2^p)]$, $[D_{2(2^p-1)}]$, $[D_{2(2^p+1)}]$, $[\mathbb{Z}_2^p\rtimes\mathbb{Z}_{2^p-1}]$, $[\mathbb{Z}_{2^p-1}]$, $[\mathbb{Z}_{2^p+1}]$, $[\mathbb{Z}_2^p]$, $[\mathbb{Z}_2^{p-1}]$, \dots, $[\mathbb{Z}_2]$ and $[1]$, we get
\begin{equation}
11\geq Con(G)\geq Con(G/\Phi(G))\geq 7+p,\nonumber 
\end{equation}and so $p=2$ or $p=3$. Note that the second situation cannot occur because $Con({\rm PSL}(2,8))=12$. Thus $p=2$, that is $G/\Phi(G)\cong A_5$. Then $\Phi(G)\neq 1$ and $r=3$, implying that $G$ has at least $12$ conjugacy classes of subgroups: the $9$ that contain $[\Phi(G)]$, two conjugacy classes of minimal subgroups that does not contain $[\Phi(G)]$ and $[1]$, a contradiction.

This completes the proof.\qed
\bigskip

Note that the proof of Theorem 1.3 also reveals two finite non-solvable groups $G$ with $Con(G)=12$, namely ${\rm SL}(2,5)$ and ${\rm PSL}(2,8)$. In both cases we have $k(G)=3$ and $r=3$.

\bigskip\noindent{\bf Acknowledgements.} The author is grateful to the reviewer for remarks which improve the previous version of the paper.

\vspace*{5ex}\small

\hfill
\begin{minipage}[t]{5cm}
Marius T\u arn\u auceanu \\
Faculty of  Mathematics \\
``Al.I. Cuza'' University \\
Ia\c si, Romania \\
e-mail: {\tt tarnauc@uaic.ro}
\end{minipage}


\begin{thebibliography}{10}
\bibitem{1} B. Baumann, {\it Endliche nichtaufl\"osbare Gruppen mit einer nilpotenten maximalen Untergruppe}, J. Algebra {\bf 38} (1976), 119-135.
\bibitem{2} V.A. Belonogov, {\it Finite groups with three classes of maximal subgroups}, Math. Sb. {\bf 131} (1986), 225-239.
\bibitem{3} V.A. Belonogov, {\it Finite groups with three classes of maximal subgroups}, Math. Sb. {\bf 131} (1986), 225-239.
\bibitem{4} A. Betz and D.A. Nash, {\it Classifying groups with a small number of subgroups}, Amer. Math. Monthly {\bf 129} (2022), 255-267.
\bibitem{5} A. Betz and D.A. Nash, {\it A note on abelian groups with fewer than $50$ subgroups}, 2020, http://dx.doi.org/10.13140/RG.2.2.24909.87520.
\bibitem{6} R. Brandl and L. Verardi, {\it Finite groups with few conjugacy classes of subgroups}, Rend. Semin. Mat. Univ. Padova {\bf 87} (1992), 267-280.
\bibitem{7} A. Das and A. Mandal, {\it Solvability of a group based on its number of subgroups}, Comm. Algebra {\bf 54} (2026), 1476-1491.
\bibitem{8} M. Herzog, {\it On finite simple groups of order divisible by three prime numbers}, J. Algebra {\bf 10} (1968), 384-388.
\bibitem{9} M. Herzog, P. Longobardi and M. Maj, {\it New criteria for solvability, nilpotency and other properties of finite groups in terms of the order elements or subgroups}, Int. J. Group Theory {\bf 12} (2023), 35-44.
\bibitem{10} B. H\"{o}fling, {\it On the number of conjugacy classes of maximal subgroups in a finite soluble group}, Arch. Math. {\bf 72} (1999), 1-8.
\bibitem{11} L. Li and H. Qu, {\it The number of conjugacy classes of non-normal subgroups of finite $p$-groups}, J. Algebra {\bf 466} (2016), 44-62.
\bibitem{12} J. Lu, L. Pang and Y. Qiu, {\it Finite groups with few non-normal subgroups}, J. Algebra Appl. {\bf 14} (2015), article ID 1550057.
\bibitem{13} A. Mednykh, {\it Counting conjugacy classes of subgroups in a finitely generated group}, J. Algebra {\bf 320} (2008), 2209-2217.
\bibitem{14} T. De Medts and M. T\u arn\u auceanu, {\it Finite groups determined by an ine\-quality of the orders of their subgroups}, Bull. Belg. Math. Soc. Simon Stevin {\bf 15} (2008), 699-704.
\bibitem{15} W. Meng and S. Li, {\it Finite groups with few conjugacy classes of non-cyclic subgroups}, Sci. Sinica Math. {\bf 44} (2014), 939-944.
\bibitem{16} D.J.S. Robinson, {\it A course in the theory of groups} (second edition), Springer-Verlag, New York, 1996.
\bibitem{17} M. T\u arn\u auceanu, {\it Finite groups determined by an inequality of the orders of their subgroups} II, Comm. Algebra {\bf 45} (2017), 4865-4868.
\end{thebibliography}
\end{document}